\documentclass[11pt]{amsart}
\usepackage{geometry}                
\usepackage{graphicx}
\usepackage{amssymb}
\usepackage{epstopdf}
\usepackage{pgf,tikz}
\usepackage{amsmath}
\usepackage{mathtools}
\newtheorem{theorem}{Theorem}[section]

\newtheorem{corollary}{Corollary}[section]

\newtheorem{lemma}{Lemma}[section]
\newtheorem{remark}{Remark}[section]
\newtheorem{example}{Example}[section]

\newtheorem{conjecture}{Conjecture}

\newcommand{\ignore}[1]{}

\newcommand{\oR}{{\mathbb R}}
\newcommand{\oN}{{\mathbb N}}

\newcommand{\MH}{{\mathcal H}}

\title[Analyse poly. opt. over simplex using multi-hypergeometric distribution]{An error analysis for polynomial optimization over the simplex based on the multivariate hypergeometric distribution}

\author{Etienne de Klerk}
\address{Tilburg University;
PO Box 90153, 5000 LE Tilburg, The Netherlands.}
\email{E.deKlerk@uvt.nl}

\author{Monique Laurent}
\address{Centrum Wiskunde \& Informatica (CWI), Amsterdam and Tilburg University;
CWI, Postbus 94079, 1090 GB Amsterdam, The Netherlands.}
\email{monique@cwi.nl}

\author{Zhao Sun}
\address{Tilburg University;
PO Box 90153, 5000 LE Tilburg, The Netherlands.}
\email{Z.Sun@uvt.nl}

\date{\today}                                           
\keywords{Polynomial optimization over the simplex, Global optimization, Nonlinear optimization}
\subjclass[2010]{90C26, 90C30}
\begin{document}
\begin{abstract}
We study the minimization of fixed-degree polynomials over the simplex. This problem is well-known to be NP-hard,
as it contains the maximum stable set problem in graph theory as a special case.
In this paper, we consider a rational approximation by taking the minimum over the regular grid,
which consists of rational points with denominator $r$ (for given $r$).
We show that the associated convergence rate is  $O(1/r^2)$ for quadratic polynomials. For general polynomials,
if there exists a rational global minimizer over the simplex, we show that the convergence rate is also of the order $O(1/r^2)$.
Our results answer a question posed by De Klerk et al. \cite{KLS13} and improves on previously known $O(1/r)$ bounds in the quadratic case.
\end{abstract}

\maketitle

\section{Introduction and preliminaries}
 We consider optimization of polynomials over the standard simplex:
\[
\Delta_n:=\left\{x\in\oR_+^n:\sum_{i=1}^nx_i=1\right\}.
\]
More precisely, given a polynomial $f\in\mathcal{H}_{n,d}$,
where $\mathcal{H}_{n,d}$ {denotes} the set of $n$-variate homogeneous real polynomials of degree $d$,
 we define
\begin{equation}\label{funderline}
\underline{f}:=\min_{x\in\Delta_n}f(x),
\end{equation}
and $\overline{f}:=\max_{x\in\Delta_n}f(x)$.
For computational complexity reasons,  we assume throughout that the polynomial $f$ has integer coefficients.

For quadratic  $f\in \MH_{n,2}$,  Vavasis \cite{SAV90} shows that problem (\ref{funderline}) admits a rational
 global minimizer $x^*$, whose bit-size is polynomial in the bit-size of the input data.
On the other hand, when the degree of $f$ is larger than $2$,  there exist polynomials $f$ for which  problem (\ref{funderline}) does
 not have any rational global minimizer.  This is the case, for instance, for  the polynomial $f(x)=2{x_1}^3-x_1\left( \sum_{i=1}^nx_i \right)^2$, whose global minimizer always has the irrational component $x_1={1/ \sqrt{6}}.$

\subsection*{Complexity and approximation results}
The global optimization problem (\ref{funderline}) is known to be NP-hard, and contain the maximum stable set problem in graphs as a special case.
Indeed, for a graph $G=(V,E)$, Motzkin and Straus \cite{MS} show that its  stability number $\alpha(G)$ can be calculated via
$${1\over \alpha(G)}=\min_{x\in \Delta_{|V|}}x^T(I+A_G)x,$$
where $I$ denotes the identity matrix and $A_G$ denotes the adjacency matrix of graph $G$.

\medskip
On the other hand, there exists a polynomial time approximation scheme (PTAS) for problem (\ref{funderline}) over
the class of polynomials $f\in \MH_{n,d}$ with fixed degree $d$, as was shown by Bomze and de Klerk   \cite{BK02}
 for degree $d=2$  and by De Klerk, Laurent and Parrilo  \cite{KLP06} for degree $d\ge 3$. The PTAS is easily described:
  It takes the minimum of $f$ over the regular grid
\[
\Delta(n,r):=\{x\in\Delta_n:rx\in\oN^n\},
\]

\noindent
for increasing values of $r\in\oN$. Note that
\begin{equation}\label{fdelta}
f_{\Delta(n,r)}:=\min_{x\in\Delta(n,r)}f(x)
\end{equation}
may be computed by performing $|\Delta(n,r)| = {n+r-1 \choose r}$ evaluations of $f$.
Thus, for fixed $r$, $f_{\Delta(n,r)}$ can be obtained in polynomial (in $n$) time. The following error
 estimates have been shown for the range $f_{\Delta(n,r)}-\underline f$ in terms of
  the range $\overline f -\underline f$ of function values.

\begin{theorem}   \cite[Theorem 3.2]{BK02} 
\label{thmklsquad}
For any polynomial $f\in \mathcal{H}_{n,2}$ and $r\ge 1$, one has
 $$f_{\Delta(n,r)}-\underline{f}\le {\overline{f}-\underline{f}\over r}.$$
\end{theorem}

\begin{theorem}\cite[Theorem 1.3]{KLP06}\label{thmklsgeneral}
For any polynomial $f\in \mathcal{H}_{n,d}$ and $r\ge 1$,  one has
\begin{equation*}
\label{final bounds}
f_{\Delta(n,r)}-\underline{f} \le 
  \left(1-{r^{\underline{d}}\over r^d}\right){2d-1\choose d}d^d(\overline{f}-\underline{f}).
\end{equation*}
\end{theorem}
For more results about the computational complexity of problem (\ref{funderline}), see \cite{EDK08,KHE08}; for
properties of the grid $\Delta(n,r)$, see \cite{Bos83}, and for recent studies of the approximation $f_{\Delta(n,r)}$, see \cite{BGY13,KLS13,SY13,ZS13,YEA12}.


De Klerk et al. \cite{KLS13} recently provided alternative proofs of the PTAS results in Theorems \ref{thmklsquad} and \ref{thmklsgeneral}.
The idea of these proofs is to define a suitable discrete probability distribution on $\Delta(n,r)$ (seen as a sample space), by using
 the multinomial distribution. (This idea is an extension of a probabilistic argument by Nesterov \cite{Nes2003}; for  the exact connection, see \cite[Section 6]{KLS13}.)

Recall that the multinomial distribution
may be explained by considering a box filled with balls of $n$ different colors, and where the fraction of balls of color $i \in \{1,\ldots,n\}$
 is denoted by $x_i$, say. If one draws $r$ balls randomly {\em with replacement} and let the random variable
  $Y_i$ denote the number of times that a ball of color $i$ was drawn, then
\[
\mathbf{Pr}\left[Y_1 = \alpha_1,\ldots,Y_n = \alpha_n\right] = {r!\over \alpha!} x^\alpha, \quad \quad \alpha \in r \Delta(n,r),
\]
where $\alpha! :=  \prod_{i=1}^n \alpha_i!$ and $x^\alpha := \prod_{i=1}^n x_i^{\alpha_i}$.
Defining the normalized random variable $X = \frac{1}{r}Y \in \Delta(n,r)$, one has
\[
\mathbb{E}[f(X)]=\sum_{\alpha\in r\Delta(n,r)}f\left({\alpha\over r}\right){r!\over \alpha!} x^\alpha,
\]
and the right-hand-side expression is precisely the \emph{Bernstein approximation of  $f$ of order $r$  at $x$}.
Therefore, since $f_{\Delta(n,r)} \le \mathbb{E}[f(X)]$, the new PTAS proof in \cite{KLS13} is essentially a consequence
 of the properties of Bernstein approximation on the standard simplex.

{This approach can be put in the more general context of the framework introduced by Lasserre \cite{Las01,Las11} based on reformulating any polynomial optimization problem as an optimization problem over measures.
When applied to our setting, this implies the following upper bound:
$$f_{\Delta(n,r)}\le \mathbb{E}_\mu(f)=\int_{\Delta(n,r)}f(x)\mu(dx)$$
for any probability measure $\mu $ on $\Delta(n,r)$.
So the work \cite{KLS13} is based on selecting the multinomial distribution with appropriate parameters as measure $\mu$.
In this paper we will select another measure, as explained below.}

\subsection*{Contribution of this paper}
In this paper, we give a partial to a question posed in  \cite{KLS13},
{concerning} the error bound in
 Theorems \ref{thmklsquad} and \ref{thmklsgeneral}, that may be rewritten as:
\begin{equation}\label{alpha}
\rho_r(f):={f_{\Delta(n,r)} - \underline f\over \overline f - \underline f} = O\left(\frac{1}{r} \right).
\end{equation}
In \cite{KLS13} several examples are given where  this error is in fact of the order $O(1/r^2)$ and
the question is posed whether this could be true in general.

Here, we give an affirmative answer for quadratic polynomials. More precisely,  we show that  $\rho_r(f)\le m/r^2$ if $f$
has a global minimizer with denominator $m$ (see Theorem \ref{thmquad}). In view of Vavasis' result \cite{SAV90} on the existence of rational minimizers
for quadratic programming, this implies that $\rho_r(f)=O(1/r^2)$ for quadratic $f$.
For polynomials $f$ of degree $d\ge 3$, when $f$ admits a rational global minimizer, we show that $\rho_r(f)=O(1/r^2)$
(see Corollaries \ref{corcub} and \ref{corgeneral}).

The main idea of our proof is to replace the multinomial distribution above by the \emph{hypergeometric distribution}, and we
therefore review some necessary background on the hypergeometric distribution next.

\subsection*{Multivariate hypergeometric distribution}
Consider a box containing $m$ balls, of which $m_i$ are of color $i$ for $i=1,\dots,n$. Thus $\sum_{i=1}^n m_i=m$.
We draw $r$ balls randomly from the box without replacement. This defines the random variable $Y_i$ as the number of balls of color $i$ in a random sample of $r$ balls.
Then, $Y=(Y_1,\ldots,Y_n)$ has the {\em  multivariate hypergeometric distribution}, with parameters $m,$ $r$ and $n$.
Given $\alpha\in \oN^n$ with $\sum_{i=1}^n\alpha_i=r$, the probability of obtaining the outcome $\alpha$, with  $\alpha_i$ balls of color $i$, is equal to
\begin{equation}\label{Yhyper}
\mathbf{Pr}\left[Y_1 = \alpha_1,\ldots,Y_n = \alpha_n\right] = {\prod_{i=1}^n {m_i\choose \alpha_i}\over {m\choose r}}.
\end{equation}

\noindent
For $\beta\in\oN^n$,  the $\beta$-th moment of the multivariate hypergeometric distribution $Y$ is  defined as
\begin{equation*}
m^{\beta}_{(n,r)}(Y):=\mathbb{E}\left( \prod_{i=1}^n{Y_i}^{\beta_i} \right)=\sum_{\alpha\in I(n,r)}\alpha^{\beta}{\prod_{i=1}^n {m_i\choose \alpha_i}\over {m\choose r}},
\end{equation*}
where $I(n,r):=\{\alpha\in \oN^n: |\alpha|:=\sum_{i=1}^n\alpha_i=r\}$. Combining \cite[relation (34.18)]{JKB97} and \cite[relation (39.6)]{JKB97}, we can obtain the explicit formula for $m^{\beta}_{(n,r)}(Y)$ in terms of the Stirling numbers of the second kind. For integers $a,b\in\oN$, the {\em Stirling number of the second kind} $S(a,b)$ counts the number of ways of partitioning a set of $a$ objects into $b$ nonempty subsets. Moreover, denote $r^{\underline{d}}:=r(r-1)\cdots(r-d+1)$.

\begin{theorem}\label{momentthm}
For $\beta\in\oN^n$, one has
\begin{equation*}
m^{\beta}_{(n,r)}(Y)=\sum_{\alpha\in\oN^n:\alpha\le\beta}{r^{\underline{|\alpha|}}\over m^{\underline{|\alpha|}}}\prod_{i=1}^n {m_i}^{\underline{\alpha_i}}S(\beta_i,\alpha_i).
\end{equation*}
\end{theorem}

Define the random variables
\begin{equation}\label{capitalX}
X=(X_1,\ldots, X_n) \ \ \text{where } X_i := Y_i/r\ \ (i=1,\ldots,n).
\end{equation}
Thus $X$ takes its values in $\Delta(n,r)$.
 Theorem \ref{momentthm} gives  the explicit formula for the moments of $X$. 
\begin{corollary}\label{momentcor}
For $\beta\in\oN^n$, one has
\begin{equation*}
m^{\beta}_{(n,r)}(X):=\mathbb{E}\left( \prod_{i=1}^n{X_i}^{\beta_i} \right)={1\over r^{|\beta|}}\sum_{\alpha\in\oN^n:\alpha\le\beta}{r^{\underline{|\alpha|}}\over m^{\underline{|\alpha|}}}\prod_{i=1}^n {m_i}^{\underline{\alpha_i}}S(\beta_i,\alpha_i).
\end{equation*}
\end{corollary}

The multivariate hypergeometric distribution can be used for upper bounding the minimum of $f$ over $\Delta(n,r)$.

\begin{lemma}\label{lemmhdpoly}
Let $f=\sum_{\beta\in I(n,d)}f_{\beta}x^{\beta}\in\mathcal{H}_{n,d}$ and let $X:=(X_1,X_2,\dots,X_n)$ be as in (\ref{Yhyper}) and  (\ref{capitalX}). Then, one has
$$f_{\Delta(n,r)}\le\mathbb{E}\left( f(X) \right),$$
and the above inequality can be strict.
\end{lemma}
\begin{proof}
By definition (\ref{capitalX}), the random variable  $X$ takes its values in $\Delta(n,r)$, which implies directly that the expected value of $f(X)$ is at least the minimum of $f$ over $\Delta(n,r)$.
In order to show the inequality can be strict, we consider the following example: $f=2x_1^2+x_2^2-5x_1x_2$. One has $\underline{f}= -{17\over 32}$ attained at the unique minimizer $({7\over 16}, {9\over 16})$. Then we let $m=16$, $m_1=7$ and $m_2=9$. When $r=2$, one can easily check that $f_{\Delta(2,2)}= -{1\over 2}$ (attained at the unique minimizer $({1\over 2}, {1\over 2})$). On the other hand, when $r=2$, $\mathbb{E}\left( f(X) \right)={31\over 80}$, and thus $\mathbb{E}\left( f(X) \right)>f_{\Delta(2,2)}$.
\qed
\end{proof}

\ignore{
\medskip
In order to upper bound the range $f_{\Delta(n,r)}-f_{\Delta(n,m)}$, by Lemma \ref{lemmhdpoly}, we can upper bound the difference $\mathbb{E}\left( f(X) \right)-f_{\Delta(n,m)}$ beforehand. This will be our strategy (see the proofs of Theorems \ref{thmquadratic}, \ref{thmcubsqf} and \ref{thmgeneral}).
}

\subsection*{Bernstein coefficients}
\noindent Any  polynomial $f=\sum_{\beta\in I(n,d)}f_{\beta}x^{\beta}\in\mathcal{H}_{n,d}$ can be  written as
\begin{equation}\label{berncoef}
f=\sum_{\beta\in I(n,d)}f_{\beta}x^{\beta}=\sum_{\beta\in I(n,d)}\left(f_{\beta}{\beta!\over d!}\right){d!\over \beta!}x^{\beta}.
\end{equation}

\noindent Then, the scalars $f_{\beta}{\beta!\over d!}$ (for $\beta\in I(n,d)$) are called the {\em Bernstein coefficients} of $f$  since  they are the coefficients of $f$ when expressing $f$ in the Bernstein basis $\{{d!\over \beta!}x^\beta: \beta\in I(n,d)\}$ of $\mathcal{H}_{n,d}$ (see e.g. \cite{KL10,KLS13,ZS13}). Combining (\ref{berncoef}) with the multinomial theorem

\begin{equation}\label{multithm}
\left(\sum_{i=1}^n x_i\right)^d=\sum_{\alpha\in I(n,d)}{d!\over \alpha!}x^{\alpha},
\end{equation}
it follows that, when $x\in \Delta_n$, $f(x)$ is a convex combination of its Bernstein coefficients $f_{\beta}{\beta!\over d!}$. Hence, for any $x\in\Delta_n$, we have
\begin{equation}\label{reprop1}
\min_{\beta\in I(n,d)}f_{\beta}{\beta!\over d!}\le f(x)\le\max_{\beta\in I(n,d)}f_{\beta}{\beta!\over d!}.
\end{equation}

\noindent In Section \ref{secgeneral}, we will make use of the following theorem by de Klerk et al. \cite{KLP06}, which bounds the range of the Bernstein coefficients in terms of the range of function values $\overline{f}-\underline{f}$.

\begin{theorem}\label{thmgene}\cite[Theorem 2.2]{KLP06}
For any polynomial $f=\sum_{\beta\in I(n,d)}f_{\beta}x^{\beta}\in\mathcal{H}_{n,d}$, one has
\begin{equation*}
\max_{\beta\in I(n,d)}f_{\beta}{\beta!\over d!}-\min_{\beta\in I(n,d)}f_{\beta}{\beta!\over d!}\le {2d-1\choose d}d^d (\overline{f}-\underline{f}).
\end{equation*}
\end{theorem}

\subsection*{Notation}
\noindent We denote $[n]:=\{1,2,\ldots,n\}$ and let $\oN^n$ be the set of all $n$-dimensional nonnegative integral vectors.
 For $\alpha\in\oN^n$, we define $|\alpha|:=\sum_{i=1}^n\alpha_i$ and $\alpha!:=\alpha_1!\alpha_2!\cdots\alpha_n!$.
 For vectors $\alpha,\beta\in\oN^n$, the inequality $\alpha\le\beta$ means $\alpha_i\le\beta_i$ for any $i\in[n]$.
As before, set $I(n,d):=\{\alpha\in \oN^n: |\alpha|=d\}$ and let $\mathcal{H}_{n,d}$ be the set of all multivariate
 real homogeneous polynomials in $n$ variables with degree $d$. Then, for $\alpha\in \oN^n$,
 we denote $x^{\alpha}:=\prod_{i=1}^nx_i^{\alpha_i}$. Similarly, for $I\subseteq[n]$,
 we let $x^{I}:=\prod_{i\in I}x_i$. A monomial $x^{\alpha}$ is called {\em square-free} (aka {\em multilinear}) if $\alpha_i\in\{0,1\}$ ($i\in[n]$), and a polynomial $f$ is called {\em square-free} if all its monomials are square-free.
Moreover, denote $x^{\underline{d}}:=x(x-1)(x-2)\cdots (x-d+1)$ for
integer $d\ge0$ and $x^{\underline{\alpha}}:=\prod_{i=1}^nx_{i}^{\underline{{\alpha}_i}}$ for $\alpha\in \oN^n$.
 We let $e$ denote  the all-ones vector and $e_i$ denote the $i$-th standard unit vector.  Furthermore, for a random variable $W$,
  $\mathbb{E}(W)$ is its expectation.

\subsection*{Structure}
The rest of the paper is organized as follows.
In Section \ref{secquad}, we consider the standard quadratic optimization
problem, while in Section \ref{seccubsqf} we treat the cubic and square-free (or \emph{multilinear}) cases.
In Section \ref{secgeneral}, we focus on the general fixed-degree polynomial optimization over the simplex.
Finally, we give all the proofs of results stated in Section \ref{seccubsqf} in the Appendix.

\section{Standard quadratic optimization}\label{secquad}

We  consider the problem ({\ref{funderline}}) where the polynomial $f$ is assumed to be  quadratic.
The following result plays a key role for our refined error analysis in Theorem \ref{thmquad} below.

\begin{theorem}\label{thmquadratic}
Let $f=x^T Qx \in \MH_{n,2}$. For any    integers $r$ and $m$   such that $1\le r\le m$,  one has
\begin{equation*}
f_{\Delta(n,r)} - f_{\Delta(n,m)} \le  \frac{m-r}{r(m-1)}\left(\overline f - \underline f\right).
\end{equation*}
\end{theorem}

\begin{proof}
Let  $x^*\in\Delta(n,m)$ be a  minimizer of $f$ over $\Delta(n,m)$, i.e.,  $f(x^*)=f_{\Delta(n,m)}$, and set $m_i= mx^*_i $ for  $i\in[n].$
If $m=1$, then $r=1$ and the result is trivial. Now assume $m\ge 2$.
Consider the random variable $X=(X_1,\ldots,X_n)$  defined as in (\ref{Yhyper}) and (\ref{capitalX}).
By Corollary \ref{momentcor},  one has
\begin{eqnarray*}
\mathbb{E}[X^2_i] &=& \left(\frac{m_i}{m}\right)^2\left(1  - \frac{m-r}{r(m-1)}+ \frac{m(m-r)}{rm_i(m-1)}\right) \ \ \ (i\in [n]), \\
\mathbb{E}[X_iX_j] &=& \frac{m_i}{m}\frac{m_j}{m}\left(1 - \frac{m-r}{r(m-1)} \right) \;\;\; (i\neq j\in [n]).
\end{eqnarray*}

Then, we have
\begin{eqnarray*}
\mathbb{E}\left[f(X)\right]
&=& \sum_{i,j\in[n]:i\neq j} Q_{ij} \mathbb{E}[X_iX_j] + \sum_{i=1}^n Q_{ii} \mathbb{E}[X_i^2]\\
&=& \sum_{i,j\in[n]:i\neq j} Q_{ij}\frac{m_i}{m}\frac{m_j}{m}\left(1 - \frac{m-r}{r(m-1)}\right)\\
&+& \sum_{i=1}^n Q_{ii}\left(\frac{m_i}{m}\right)^2\left(1  - \frac{m-r}{r(m-1)}+ \frac{m(m-r)}{rm_i(m-1)}\right)\\
&=& \sum_{i,j\in[n]} Q_{ij}x_i^*x_j^*\left(1 - \frac{m-r}{r(m-1)}\right) +  \frac{m-r}{r(m-1)}\sum_{i=1}^n Q_{ii}x_i^*\\
&\le& f(x^*) - \frac{m-r}{r(m-1)}\underline{f} +  \frac{m-r}{r(m-1)}\max_{i\in[n]} Q_{ii} \\
&\le& f(x^*) - \frac{m-r}{r(m-1)}\underline{f}  +  \frac{m-r}{r(m-1)}\overline f.
\end{eqnarray*}

Hence, we obtain
$$\mathbb{E}\left[f(X)\right]-f_{\Delta(n,m)}=\mathbb{E}\left[f(X)\right]-f(x^*) \le \frac{m-r}{r(m-1)}(\overline f-\underline{f}).$$

Using Lemma \ref{lemmhdpoly}, we can conclude the proof.
\qed
\end{proof}

\vspace{0.5cm}
When $f$ is quadratic, Vavasis \cite{SAV90} shows that there always exists a rational global minimizer $x^*$ for problem (\ref{funderline}). Say, $x^*$ has denominator $m$, i.e., $x^*\in \Delta(n,m)$. Our next result gives an upper bound for the error estimate $f_{\Delta(n,r)}-\underline f$, in terms of this denominator $m$.

 \begin{theorem}\label{thmquad}
 Let $f=x^T Qx \in \MH_{n,2}$, and let $x^*$ be a global minimizer of $f$ over $\Delta_n$, with denominator $m$.
For any    integer $r\ge 1$,  one has
\begin{equation*}
f_{\Delta(n,r)} - \underline f \le  \frac{m}{r^2}\left(\overline f - \underline f\right).
\end{equation*}
\end{theorem}

The proof uses the following easy fact (whose proof is omitted).

\begin{lemma}\label{lemkmr}
Let  $r,k,m\ge 1$ be integers such that  $(k-1)m< r\le km$. Then, $${km-r\over km-1}\le {m\over r}.$$
\end{lemma}

\begin{proof} {\em (Proof of Theorem \ref{thmquad})}
Let $k\ge 1$ be an integer such that $(k-1)m <r \le km$.
We apply Theorem \ref{thmquadratic} to $r$ and $km$ (instead of $m$) and obtain that
$$f_{\Delta(n,r)}-f_{\Delta(n,km)}\le {km-r\over r(km-1)} (\overline f- \underline f).$$
Now, observe that $f_{\Delta(n,km)}=f_{\Delta(n,m)}=\underline f$, since $x^*\in \Delta(n,m)\subseteq \Delta(n,km)\subseteq \Delta_n$, and use the inequality from Lemma \ref{lemkmr}.
\qed\end{proof}

\medskip
As a direct  application of Theorem \ref{thmquad}, we see that the rate of convergence of the sequence $\rho_r(f)$ in (\ref{alpha}) 
 is in the order $O(1/r^2)$, where the constant depends only on  the denominator  of a rational global minimizer.

\begin{corollary}\label{corquad}
For any quadratic polynomial $f\in \MH_{n,2}$,
$\rho_r(f)=O({1/r^2})$.
\end{corollary}

Moreover, the results of Theorems \ref{thmquadratic} and \ref{thmquad}
refine the known error estimate  from Theorem \ref{thmklsquad}, which shows that $\rho_r(f)\le {1\over r}$.  
To see it, use Theorem \ref{thmquadratic} and  the fact that ${m-r\over r(m-1)}\le {1\over r}$ if $1\le r\le m$, and use Theorem \ref{thmquad} and the inequality ${m\over r^2}\le {1\over r}$ in the case $r\ge m$.

\medskip
The following example shows that the inequality in Theorem \ref{thmquadratic} can be tight.

\begin{example}\cite[Example 2]{KLS13}
Consider the quadratic polynomial $f=\sum_{i=1}^nx_i^2$. Since $f$ is convex, one can easily check $\overline{f}=1$ (attained at any standard unit vector) and $\underline{f}={1\over n}$ (attained at $x={1\over n}e$, with denominator $m=n$).
 Moreover, for any integer $r\le n$, we have $f_{\Delta(n,r)}={1\over r}$. Thus, we have
$$f_{\Delta(n,r)}-\underline{f}={n-r\over r(n-1)}(\overline{f}-\underline{f})= {m-r\over r(m-1)}(\overline{f}-\underline{f}).$$
Hence, for this example,  the result in Theorem \ref{thmquadratic} is tight, while the result in Theorem \ref{thmklsquad} is not tight.
\end{example}

\section{Cubic and square-free polynomial optimizations over the simplex}\label{seccubsqf}

For the minimization of cubic and square-free polynomials over the standard simplex, the following results from  \cite{KLS13} refine Theorem \ref{thmklsgeneral}.

\begin{theorem}\label{thmklscubsqf}
\begin{itemize}
\item[(i)]
\cite[Corollary 2]{KLS13}
For any polynomial $f\in \mathcal{H}_{n,3}$ and $r\ge 2$, one has $$f_{\Delta(n,r)}-\underline{f}\le \left( {4\over r}-{4\over r^2} \right)\left(\overline{f}-\underline{f}\right).$$
\item[(ii)]\cite[Corollary 3]{KLS13}
For any square-free polynomial $f\in \MH_{n,d}$ and $r\ge 1$, one has
\begin{equation*}
f_{\Delta(n,r)}-\underline{f}\le \left( 1-{r^{\underline{d}}\over r^d} \right)\left(\overline f - \underline f\right).
\end{equation*}
\end{itemize}
\end{theorem}

We can show the following analogue of Theorem \ref{thmquadratic} for cubic and square-free polynomials. We delay the proof to Appendix \ref{appthm}, since the details are similar to the quadratic case (but more technical).

\begin{theorem}\label{thmcubsqf}
\begin{itemize}
\item[(i)]
Let $f\in \MH_{n,3}$. Given  integers $r,m$ satisfying $1\le r\le m$ and $m\ge 3$, one has
\begin{equation*}
f_{\Delta(n,r)} - f_{\Delta(n,m)} \le \frac{(m-r)(4mr-2m-2r)}{r^2(m-1)(m-2)}\left(\overline f - \underline f\right).
\end{equation*}
\item [(ii)]
 Let $f\in \MH_{n,d}$ be a  square-free polynomial. 
Given integers $r,m$ satisfying $1\le r\le m$ and $m\ge d$, one has
\begin{equation*}
f_{\Delta(n,r)}-f_{\Delta(n,m)}\le \left( 1-{r^{\underline{d}}\over r^d}{m^d\over m^{\underline{d}}} \right)\left(\overline f - \underline f\right).
\end{equation*}
\end{itemize}
\end{theorem}

When problem (\ref{funderline}) admits a rational global minimizer, then one can show that Theorem \ref{thmcubsqf} (ii) implies Theorem \ref{thmklscubsqf} (ii), and that Theorem \ref{thmcubsqf} (i) implies Theorem \ref{thmklscubsqf} (i) for $r\ge 1+{m-1\over \sqrt{2m}-1}$. We give the proofs for these statements in Appendix \ref{appimply}.

\ignore{
the above results imply the following error bounds from \cite{KLS13} for the range $f_{\Delta(n,r)}-\underline f$ (for the cubic case, the implication holds when $r\ge 1+{m-1\over \sqrt{2m}-1}$). For the sake of completeness, we give the proof for this statement in Appendix \ref{appimply}.

\begin{theorem}\label{thmklscubsqf}
\begin{itemize}
\item[(i)]
\cite[Corollary 2]{KLS13}
For any polynomial $f\in \mathcal{H}_{n,3}$ and $r\ge 2$, one has $$f_{\Delta(n,r)}-\underline{f}\le \left( {4\over r}-{4\over r^2} \right)\left(\overline{f}-\underline{f}\right).$$
\item[(ii)]\cite[Corollary 3]{KLS13}
For any square-free polynomial $f\in \MH_{n,d}$ and $r\ge 1$, one has
\begin{equation*}
f_{\Delta(n,r)}-\underline{f}\le \left( 1-{r^{\underline{d}}\over r^d} \right)\left(\overline f - \underline f\right).
\end{equation*}
\end{itemize}
\end{theorem}
}

\medskip
Theorem \ref{thmklscubsqf} shows that the ratio $\rho_r(f)$ is in the order $O(1/r)$. As an application of Theorem \ref{thmcubsqf}, we can show that  the ratio $\rho_r(f)$ is in the order $O(1/r^2)$ for cubic polynomials admitting a rational global minimizer over the simplex (see Corollary \ref{corcub}, whose proof is given in Appendix \ref{appcor}). The same holds for square-free polynomials as we will  see in the next section.

\begin{corollary}\label{corcub}
Let $f\in \MH_{n,3}$ and assume that $f$ has  a rational global minimizer in $\Delta_n$.
Then, $\rho_r(f)=O(1/r^2)$.
\end{corollary}

\section{General fixed-degree polynomial optimization over the simplex}\label{secgeneral}

In this section, we study the general fixed-degree polynomial optimization problem over the standard simplex. We first  upper bound  the range $f_{\Delta(n,r)}-f_{\Delta(n,m)}$ in terms of $\overline{f}-\underline{f}$.

\begin{theorem}\label{thmgeneral}
Let $f\in\mathcal{H}_{n,d}$. For any integers $r,m$ satisfying $1\le r\le m$ and $m\ge d$,
one has
\begin{eqnarray*}
f_{\Delta(n,r)}-f_{\Delta(n,m)}\le \left(1-{r^{\underline{d}}m^d\over r^dm^{\underline{d}}}\right){2d-1\choose d}d^d (\overline{f}-\underline{f}).
\end{eqnarray*}
\end{theorem}

\noindent Note that when $f$ is square free, we have proved a better bound in Theorem \ref{thmcubsqf} (ii).

\noindent For the proof of Theorem \ref{thmgeneral}, we will use the following Vandermonde-Chu identity
\begin{equation}\label{vcid}
\left(\sum_{i=1}^n x_i\right)^{\underline{d}}=\sum_{\alpha\in I(n,d)}{d!\over \alpha!}x^{\underline{\alpha}}\ \ \ \ \forall x\in\oR^n\ \
\end{equation}
{(see \cite{PR01})}, as well as the multinomial theorem (\ref{multithm}).
We will also need the following two lemmas about the Stirling numbers of the second kind.
\begin{lemma}\label{lemsn1}(e.g. \cite[Lemma 3]{KLS13})
For any positive integer $d$ and $r\ge 1$, one has
\begin{equation*}\label{reptas2}
\sum_{k=1}^{d-1}r^{\underline{k}}S(d,k)=r^d-r^{\underline{d}}.
\end{equation*}
\end{lemma}

\begin{lemma}\label{lemsn2}(e.g. \cite[Lemma 4]{KLS13})
Given $\alpha\in I(n,k)$ and $d>k$, one has
\begin{equation*}\label{reptas3}
S(d,k)={\alpha!\over k!}\sum_{\beta\in I(n,d)}{d!\over\beta!}\prod_{i=1}^n S(\beta_i,\alpha_i).
\end{equation*}
\end{lemma}

Furthermore, we will use the following technical result.
\begin{lemma}\label{lemthec}
Given $\beta\in I(n,d)$, for any integers $r,m$ with $1\le r\le m$, $m\ge d$ and integers $m_i$ $(i\in[n])$ with $\sum_{i=1}^nm_i=m$, one has
\begin{eqnarray}
A_\beta:=r^{\underline{d}}\left(\prod_{i=1}^n {m_i}^{\underline{\beta_i}}- \prod_{i=1}^n {m_i}^{\beta_i} \right)&+&\sum_{\alpha\in\oN^n:\alpha\le \beta,\alpha\neq\beta}{r^{\underline{|\alpha|}}m^{\underline{d}}\over m^{\underline{|\alpha|}}}\prod_{i=1}^n{m_i}^{\underline{\alpha_i}}S(\beta_i,\alpha_i)\ge0,\label{defab}\\
\sum_{\beta\in I(n,d)}{d!\over \beta!}A_{\beta}&=&r^dm^{\underline{d}}-r^{\underline{d}}m^d.\label{claimab}
\end{eqnarray}
\end{lemma}

\begin{proof}
We first prove (\ref{defab}).
For any $\alpha\in\oN^n$ with $|\alpha|\le d$, one can easily check that ${r^{\underline{|\alpha|}}\over r^{\underline{d}}}\ge{m^{\underline{|\alpha|}}\over m^{\underline{d}}}$, that is, $r^{\underline{d}}\le {r^{\underline{|\alpha|}}m^{\underline{d}}\over m^{\underline{|\alpha|}}}.$ Hence, one has
\begin{eqnarray*}
A_\beta&=&r^{\underline{d}}\left(\prod_{i=1}^n {m_i}^{\underline{\beta_i}}- \prod_{i=1}^n {m_i}^{\beta_i} \right)+\sum_{\alpha\in\oN^n:\alpha\le \beta,\alpha\neq\beta}{r^{\underline{|\alpha|}}m^{\underline{d}}\over m^{\underline{|\alpha|}}}\prod_{i=1}^n{m_i}^{\underline{\alpha_i}}S(\beta_i,\alpha_i)\\
&\ge&
r^{\underline{d}}
\underbrace{\left(\prod_{i=1}^n {m_i}^{\underline{\beta_i}}- \prod_{i=1}^n {m_i}^{\beta_i} +
\sum_{\alpha\in\oN^n:\alpha\le \beta,\alpha\neq\beta}\prod_{i=1}^n{m_i}^{\underline{\alpha_i}}S(\beta_i,\alpha_i)\right)}_{:=B_\beta}
= r^{\underline d} B_\beta.\\
\end{eqnarray*}

Then we consider the quantity $B_\beta$ and show that $B_\beta=0$.
As $S(\beta_i,\beta_i)=1$, one can rewrite $B_\beta$ as
\begin{equation*}
{B_\beta }= \sum_{\alpha\in\oN^n:\alpha\le \beta}\prod_{i=1}^n{m_i}^{\underline{\alpha_i}}S(\beta_i,\alpha_i) - \prod_{i=1}^n {m_i}^{\beta_i}.
\end{equation*}

Applying Lemma \ref{lemsn1} (with $(m_i,\beta_i)$ in place of $(r,d)$), we  have
$${m_i}^{\beta_i}=\sum_{\alpha_i=0}^{\beta_i}{m_i}^{\underline{\alpha_i}}S(\beta_i,\alpha_i), \
\text{ implying } \
\prod_{i=1}^n {m_i}^{\beta_i}=\sum_{\alpha\in\oN^n:\alpha\le \beta}\prod_{i=1}^n{m_i}^{\underline{\alpha_i}}S(\beta_i,\alpha_i),$$
which shows that $B_\beta=0$, and thus $A_\beta\ge 0$, which concludes the proof of (\ref{defab}).

We now show (\ref{claimab}). 
 By the definition (\ref{defab}), one has
\begin{eqnarray*}\label{re5}
\sum_{\beta\in I(n,d)}{d!\over \beta!}A_{\beta}
&=& \underbrace{\sum_{\beta\in I(n,d)}{d!\over \beta!}r^{\underline{d}}\left(\prod_{i=1}^n {m_i}^{\underline{\beta_i}}- \prod_{i=1}^n {m_i}^{\beta_i} \right)}_{:=C_1}\\
&+& \underbrace{\sum_{\beta\in I(n,d)}{d!\over \beta!}\sum_{\alpha\in\oN^n:\alpha\le \beta,\alpha\neq\beta}{r^{\underline{|\alpha|}}m^{\underline{d}}\over m^{\underline{|\alpha|}}}\prod_{i=1}^n{m_i}^{\underline{\alpha_i}}S(\beta_i,\alpha_i)}_{:=C_2}.
\end{eqnarray*}

On the one hand, using the  Vandermonde-Chu identity (\ref{vcid}), the multinomial theorem (\ref{multithm}) and the identity $\sum_{i=1}^nm_i=m$, we find
\begin{eqnarray*}
C_1 = r^{\underline{d}}(m^{\underline{d}}-m^d).
\end{eqnarray*}

On the other hand, exchanging the summations in the definition of $C_2$, one obtains
\begin{eqnarray*}
C_2&=& m^{\underline{d}}\sum_{k=1}^{d-1}\sum_{\alpha\in I(n,k)}{r^{\underline{|\alpha|}}\over m^{\underline{|\alpha|}}}\prod_{i=1}^n{m_i}^{\underline{\alpha_i}} \left(\sum_{\beta\in I(n,d)}{d!\over\beta!}\prod_{i=1}^nS(\beta_i,\alpha_i)\right)\\
&=& m^{\underline{d}}\sum_{k=1}^{d-1}{r^{\underline{k}}\over m^{\underline{k}}}S(d,k)\left(\sum_{\alpha\in I(n,k)}{k!\over\alpha!}\prod_{i=1}^n{m_i}^{\underline{\alpha_i}}\right) \hspace*{2.3cm} \text{[using Lemma \ref{lemsn2}]}\\
&=& m^{\underline{d}}\sum_{k=1}^{d-1}r^{\underline{k}}S(d,k) \hspace*{3.1cm}\text{[using Vandermonde-Chu identity (\ref{vcid})]}\\
&=& m^{\underline{d}}(r^d-r^{\underline{d}}) \hspace*{7.1cm}\text{[using Lemma \ref{lemsn1}]}
\end{eqnarray*}

We can now conclude that  $\sum_{\beta\in I(n,d)}{d!\over \beta!}A_{\beta}=C_1+C_2=r^dm^{\underline{d}}-r^{\underline{d}}m^d$.
\qed
\end{proof}

\vspace{0.5cm}
Now we are ready to prove Theorem \ref{thmgeneral}.

\begin{proof}{\em (of Theorem \ref{thmgeneral})}
Let $x^*\in\Delta(n,m)$ be a minimizer of $f$ over $\Delta(n,m)$, i.e.,  $f(x^*)=f_{\Delta(n,m)}$. Set $m_i= mx^*_i$ for $i\in[n].$
Let the random variables $X_i$ be defined as in (\ref{Yhyper}) and (\ref{capitalX}), so that the random variable  $X=(X_1,X_2,\dots,X_n)$ takes its values in $\Delta(n,r)$.
By  Corollary \ref{momentcor} we have:
\begin{eqnarray*}
\mathbb{E}[X^{\beta}]={1\over r^{d}}\sum_{\alpha\in\oN^n:\alpha\le\beta}{r^{\underline{|\alpha|}}\over m^{\underline{|\alpha|}}}\prod_{i=1}^n {m_i}^{\underline{\alpha_i}}S(\beta_i,\alpha_i).
\end{eqnarray*}

Then, as $S(\beta_i,\beta_i)=1$, we can rewrite
\begin{eqnarray*}
\mathbb{E}[X^{\beta}]&=& {1\over r^{d}}{r^{\underline{d}}\over m^{\underline{d}}}\prod_{i=1}^n {m_i}^{\underline{\beta_i}}+\underbrace{{1\over r^{d}}\sum_{\alpha\in\oN^n:\alpha\le\beta,\alpha\ne \beta}{r^{\underline{|\alpha|}}\over m^{\underline{|\alpha|}}}\prod_{i=1}^n {m_i}^{\underline{\alpha_i}}S(\beta_i,\alpha_i)}_{:=D_\beta} \\
&=&\prod_{i=1}^n \left({m_i\over m}\right)^{\beta_i}\left[ {r^{\underline{d}}\over r^d}{m^d\over m^{\underline{d}}}+{r^{\underline{d}}\over r^d}{m^d\over m^{\underline{d}}}\left(\prod_{i=1}^n{{m_i}^{\underline{\beta_i}}\over {m_i}^{\beta_i}}-1\right)\right]+D_\beta\\
&=& \underbrace{\prod_{i=1}^n \left({m_i\over m}\right)^{\beta_i}{r^{\underline{d}}\over r^d}{m^d\over m^{\underline{d}}}}_{:=T_1}+\underbrace{{r^{\underline{d}}\over r^d m^{\underline{d}}}\left(\prod_{i=1}^n {m_i}^{\underline{\beta_i}}- \prod_{i=1}^n {m_i}^{\beta_i} \right)+D_\beta}_{:=T_2} \\
&=& T_1+T_2 = (x^*)^{\beta}{r^{\underline{d}}m^d\over r^dm^{\underline{d}}}
+{A_{\beta}\over r^dm^{\underline{d}}}.
\end{eqnarray*}
For the above  last equality,
note that, since $x_i^*={m_i\over m}$, one has
$$T_1=(x^*)^{\beta}{r^{\underline{d}}m^d\over r^dm^{\underline{d}}}$$
and
 using the definition of $A_{\beta}$ in (\ref{defab}), we can write
$$T_2={A_{\beta}\over r^dm^{\underline{d}}}.$$

Thus we obtain
\begin{eqnarray*}
\mathbb{E}[f(X)]&=& \mathbb{E}[\sum_{\beta\in I(n,d)}f_{\beta}X^{\beta}]=\sum_{\beta\in I(n,d)}f_{\beta}\mathbb{E}[X^{\beta}] \\
&=&{r^{\underline{d}}\over r^d}{m^d\over m^{\underline{d}}}f(x^*)
+{1\over r^dm^{\underline{d}}}\sum_{\beta\in I(n,d)}f_{\beta}A_{\beta}.
\end{eqnarray*}

Therefore, we have
\begin{eqnarray*}
r^dm^{\underline{d}}(\mathbb{E}[f(X)]-f(x^*))=(r^{\underline{d}}m^d-r^dm^{\underline{d}})f(x^*)
+\sum_{\beta \in I(n,d)} f_\beta A_\beta.
\end{eqnarray*}

We now upper bound the two terms $(r^{\underline{d}}m^d-r^dm^{\underline{d}})f(x^*)$ and
$\sum_{\beta \in I(n,d)} f_\beta A_\beta$.

First,
since $r^{\underline{d}}m^d-r^dm^{\underline{d}}<0$ and $f(x^*)\ge \min_{\beta\in I(n,d)}f_{\beta}$ (see (\ref{reprop1})), one obtains
\begin{eqnarray}
(r^{\underline{d}}m^d-r^dm^{\underline{d}})f(x^*) \le
(r^{\underline{d}}m^d-r^dm^{\underline{d}})\left(\min_{\beta\in I(n,d)}f_{\beta}{\beta!\over d!}\right).\label{re11}
\end{eqnarray}

Second, using the fact that  $A_\beta\ge 0$ (by Lemma \ref{lemthec}), 
one obtains
\begin{eqnarray*}
\sum_{\beta\in  I(n,d)} f_\beta A_\beta \le 
 \left(\max_{\beta\in I(n,d)}f_{\beta}{\beta!\over d!}\right)\sum_{\beta\in I(n,d)}{d!\over \beta!}A_{\beta}.\label{re12}
\end{eqnarray*}

Using the identity $\sum_{\beta\in I(n,d)}{d!\over \beta!}A_{\beta}=r^dm^{\underline{d}}-r^{\underline{d}}m^d$ (see (\ref{claimab})), one can obtain
\begin{eqnarray*}
\sum_{\beta\in I(n,d)}f_\beta A_\beta \le\left(\max_{\beta\in I(n,d)}f_{\beta}{\beta!\over d!}\right)(r^dm^{\underline{d}}-r^{\underline{d}}m^d).
\end{eqnarray*}

Combining with (\ref{re11}), this implies
\begin{eqnarray*}
r^dm^{\underline{d}}(\mathbb{E}[f(X)]-f(x^*))
\le(r^dm^{\underline{d}}-r^{\underline{d}}m^d)\left(\max_{\beta\in I(n,d)}f_{\beta}{\beta!\over d!}-\min_{\beta\in I(n,d)}f_{\beta}{\beta!\over d!}\right).
\end{eqnarray*}

Using Theorem \ref{thmgene}, Lemma \ref{lemmhdpoly} and the fact that $f(x^*)=f_{\Delta(n,m)}$, we finally obtain
\begin{eqnarray*}
r^dm^{\underline{d}} (f_{\Delta(n,r)}-f_{\Delta(n,m)}) \le
r^dm^{\underline{d}}(\mathbb{E}[f(X)]-f(x^*))
\le(r^dm^{\underline{d}}-r^{\underline{d}}m^d){2d-1\choose d}d^d (\overline{f}-\underline{f}),
\end{eqnarray*}
which concludes the proof of Theorem \ref{thmgeneral}.
\qed
\end{proof}

\vspace{0.5cm}
In what follows we  now assume that $f\in \MH_{n,d}$ has a rational global minimizer $x^*$ with denominator $m$, i.e., $x^*\in \Delta(n,m)$, so that
$\underline f=f_{\Delta(n,m)}$.

First, observe that  Theorem \ref{thmgeneral} refines the   result from Theorem \ref{thmklsgeneral} (which follows from the fact that
$1-{r^{\underline d}(km)^d\over r^d (km)^{\underline d}}\le 1-{r^{\underline d}\over r^d}$ for any $k\ge1$).

Next, we show as an application of   Theorem \ref{thmgeneral}  that the ratio $\rho_r(f)$ is in the order $O(1/r^2)$.

\begin{corollary}\label{corgeneral}
Let $f\in \MH_{n,d}$ and assume that there exists  a rational global minimizer for problem (\ref{funderline}). Then,
$\rho_r(f)=O(1/r^2)$.
\end{corollary}

\noindent For the proof of Corollary \ref{corgeneral}, we need the following notation.
Consider the univariate polynomial $(x-1)(x-2)\cdots(x-d+1)$ (in the variable $x$), which can be written as
\begin{equation}\label{def1}
(x-1)(x-2)\cdots(x-d+1)=x^{d-1}-a_{d-2}x^{d-2}+a_{d-3}x^{d-3}+\cdots+(-1)^{d-1}a_0=x^{d-1}+p(x),
\end{equation}
setting
\begin{equation}\label{def2}
p(x)=\sum_{i=0}^{d-2}(-1)^{d-1-i}a_ix^i,
\end{equation}
where $a_i$ are positive integers depending only on $d$ for any $i\in\{0,1,\dots,d-2\}$.
We also need the following lemma. 

\begin{lemma}\label{lemsigma}
Let $r,m$ and $k$ be integers satisfying $m\ge d$ ,$k\ge 1$ and $(k-1)m< r\le km$. Then one has
$$1-{r^{\underline{d}}(km)^d\over r^d(km)^{\underline{d}}}\le {m\over r^2}c_d,$$
for some constant $c_d$ depending only on $d$.
\end{lemma}

\begin{proof}
Based on (\ref{def2}), one can write
\begin{equation*}
1-{r^{\underline{d}}(km)^d\over r^d(km)^{\underline{d}}}=\underbrace{{(km)^{d-1}\over (km-1)(km-2)\cdots (km-d+1)}}_{:=\sigma_0(r,km)}\underbrace{\left[ {p(km)\over (km)^{d-1}}-{p(r)\over r^{d-1}} \right]}_{:=\sigma_1(r,km)}.
\end{equation*}

\noindent
First we consider the term $\sigma_{0}(r,km)$. For any  integer $i\in \{1,\ldots,d-1\}$, as $k\ge 1$ and $m\ge d$, we have that $km(d-1)\ge id$, which implies ${km\over km-i}\le d$. Hence, one has $\sigma_0(r,km)\le d^{d-1}$.

\medskip
\noindent Next we consider the term $\sigma_{1}(r,km)$. Recalling (\ref{def1}), we can write $\sigma_1(r,km)$ as $\sigma_1(r,km)=\sum_{i=0}^{d-2}(-1)^{d-1-i}a_i\left({1\over (km)^{d-1-i}}-{1\over r^{d-1-i}}\right)$. Since $r\le km$, then ${1\over km}\le {1\over r}$ and ${1\over (km)^{d-1-i}}\le {1\over r^{d-1-i}}$ for any $i\in\{0,1,\dots,d-2\}$. This gives:

\begin{equation}\label{eqsigma}
\sigma_{1}(r,km)\le \sum_{i=0}^{d-2}a_i\left({1\over r^{d-1-i}}-{1\over (km)^{d-1-i}}\right).
\end{equation}

\noindent Then, we consider the term ${1\over r^{d-1-i}}-{1\over (km)^{d-1-i}}$ (for any $i\in\{0,1,\dots,d-2\}$) in (\ref{eqsigma}). For any integer $s\in[d-1]$, we have
$${1\over r^s}-{1\over (km)^s}={(km)^s-r^s\over r^s(km)^s}=D_1\cdot D_2,$$
setting
\begin{eqnarray*}
D_1&=& {km-r\over kmr},\\
D_2&=&{(km)^{s-1}+(km)^{s-2}r+\cdots+r^{s-1}\over r^{s-1}(km)^{s-1}} .
\end{eqnarray*}

\noindent On the one hand, one has $D_1\le {km-r\over r(km-1)}\le {m\over r^2}$, where the second inequality follows by Lemma \ref{lemkmr}.
On the other hand, observe that for any $i,j\in\{0,1,\dots,s-1\}$ with $i+j=s-1$, one has $(km)^ir^j\le(km)^{s-1}r^{s-1}$. Hence,
$D_2\le s\le d-1.$
That is, $${1\over r^s}-{1\over (km)^s}\le {m(d-1)\over r^2}.$$

\noindent Using this in  (\ref{eqsigma}), we find that  $\sigma_{1}(r,km)\le {m(d-1)\over r^2}\sum_{i=0}^{d-2}a_i$. From (\ref{def1}) and (\ref{def2}), we know that the term $(d-1)(\sum_{i=0}^{d-2}a_i)$ is a constant $c_d$ that depends only on $d$. This concludes the proof.
\qed
\end{proof}
\vspace{0.5cm}

\noindent We can now prove Corollary \ref{corgeneral}.

\begin{proof}{\em (proof of Corollary \ref{corgeneral})}
Let $x^*\in \Delta(n,m)$ be a rational global minimizer of $f$ over $\Delta_n$.
Let $r\ge d$ and let $k\ge 1$ be an integer such that $(k-1)m<r\le km$.
\ignore{
\noindent (i) If $d\le r\le m$,
then Theorem \ref{thmgeneral} reads
$$f_{\Delta(n,r)}-f_{\Delta(n,m)}\le \left(1-{r^{\underline{d}}m^d\over r^dm^{\underline{d}}}\right){2d-1\choose d}d^d (\overline{f}-\underline{f}).$$

\noindent One can easily conclude $\alpha_r(f)=O(1/r^2)$ using Lemma \ref{lemsigma} (setting $k=1$).

\noindent (ii) If $(k-1)m < r\le km$ with positive integer $k\ge 2$, then
}
Using  Theorem \ref{thmgeneral} (applied to $r$ and $km$ (instead of $m$)), we obtain  that
$$f_{\Delta(n,r)}-\underline{f}=f_{\Delta(n,r)}-f_{\Delta(n,km)}\le \left(1-{r^{\underline{d}}(km)^d\over r^d(km)^{\underline{d}}}\right){2d-1\choose d}d^d (\overline{f}-\underline{f}).$$
Combining with  Lemma \ref{lemsigma}, one can conclude.
\qed
\end{proof}

\section{Concluding remarks}
As explained in the introduction, the analysis presented here is essentially a modification of {the} analysis in \cite{KLS13},
in the sense that one discrete distribution on $\Delta(n,r)$ is replaced by another.

Having said that, the analysis in the current paper does not imply the PTAS results in \cite{KLS13} for non-quadratic $f$, due to
the restrictive assumption of a rational global minimizer.
It is not clear at this time if this assumption is an artefact of our analysis using the hypergeometric distribution, or if there
exist examples of problem {(\ref{funderline})} where all global minimizers are irrational and $\rho_r(f) = \Omega(1/r)$.
This remains as an interesting question for future research.

\begin{appendix}

\section{}\label{appthm}

We  give here the proof of  Theorem \ref{thmcubsqf}. As in the proof of Theorem \ref{thmquadratic}, let $x^*\in\Delta(n,m)$ be a minimizer of $f$ over $\Delta(n,m)$, i.e.,  $f(x^*)=f_{\Delta(n,m)}$, and set  $m_i= mx^*_i$ for $i\in[n].$
Consider  the random variables $X_i$  defined in (\ref{Yhyper}) and (\ref{capitalX}), so that  $X=(X_1,X_2,\dots,X_n)$ takes its values in $\Delta(n,r)$.

\medskip
First we consider the case (i) when $f$ is a homogeneous polynomial of degree 3. Write $f$  as
\begin{equation*}\label{cubicf}
f=\sum_{i=1}^nf_ix_i^3+\sum_{1\le i<j\le n}(f_{ij}x_ix_j^2+g_{ij}x_i^2x_j)+\sum_{1\le i<j<k\le n}f_{ijk}x_ix_jx_k.
\end{equation*}
By Corollary \ref{momentcor}, for any $i,j,k \in[n]$, one has
\begin{eqnarray*}
&&\mathbb{E}[X^3_i] =\left({m_i\over m}\right)^3\bigg[ 1-(m-r){3mr-2(m+r)\over r^2(m-1)(m-2)}
+(m-r){3rm_im^2-3m_im^2+m^3-2rm^2\over r^2m_i^2(m-1)(m-2)} \bigg]\\
&&\mathbb{E}[X_i^2X_j] = \left({m_i\over m}\right)^2{m_j\over m} \bigg[ 1-(m-r){3mr-2(m+r)\over r^2(m-1)(m-2)}
+(m-r){(r-1)m^2 \over r^2m_i(m-1)(m-2)} \bigg]\\
&&\mathbb{E}[X_iX_jX_k] = {m_i\over m}{m_j\over m}{m_k\over m}\left[ 1-(m-r){3mr-2(m+r)\over r^2(m-1)(m-2)} \right].
\end{eqnarray*}
Therefore, one obtains
\begin{equation}\label{re0}
\begin{array}{lll}
\mathbb{E}\left[f(X)\right] &=& \sum_if_i\mathbb{E}[X^3_i]+\sum_{i<j}\left(f_{ij}\mathbb{E}[X_iX_j^2] +g_{ij}\mathbb{E}[X_i^2X_j]\right)+\sum_{i<j<k}f_{ijk}\mathbb{E}[X_iX_jX_k]\\
&=& f(x^*)\left[ 1-(m-r){3mr-2(m+r)\over r^2(m-1)(m-2)} \right]+{m-r\over r^2(m-1)(m-2)}\sigma,\\
\end{array}
\end{equation}
where we set
\begin{equation}\label{re2}
\sigma :=
 \sum_{i=1}^nf_i{m_i\over m}(3m_ir-3m_i+m-2r) +m(r-1)\sum_{i<j}(f_{ij}+g_{ij}){m_i\over m}{m_j\over m}.
\end{equation}
As in \cite{KLP06}, by evaluating
 $f$ at $e_i$ and $(e_i+e_j)/2$,  we obtain respectively the relations:
\begin{eqnarray}
\underline{f}\le f_i\le\overline{f},\label{refi}\\
f_i+f_j+f_{ij}+g_{ij}\le 8\overline{f}.\label{refij}
\end{eqnarray}
Using (\ref{refij}), we obtain
$$\sum_{i<j}(f_{ij}+g_{ij}){m_i\over m}{m_j\over m}\le \sum_{i<j}(8\overline{f}-f_i-f_j){m_i\over m}{m_j\over m}
=8\overline{f}\sum_{i<j}{m_i\over m}{m_j\over m}-\sum_{i=1}^nf_i{m_i\over m}\left( 1-{m_i\over m} \right).$$
We use this inequality together with (\ref{refi})  to upper bound the term  $\sigma$ from
 (\ref{re2}):
\begin{eqnarray*}
\sigma
&\le&\sum_{i=1}^nf_i{m_i\over m}(4m_ir-4m_i+2m-2r-mr)
+8m(r-1)\overline{f}\sum_{i<j}{m_i\over m}{m_j\over m}  \nonumber\\
&=& 4m(r-1)\left( \sum_{i=1}^nf_i\left({m_i\over m}\right)^2+2\overline{f}\sum_{i<j}{m_i\over m}{m_j\over m} \right)+(2m-2r-mr)\sum_{i=1}^nf_i{m_i\over m}\\
&\le &4m(r-1)\overline{f}\left( \sum_{i=1}^n\left({m_i\over m}\right)^2+2\sum_{i<j}{m_i\over m}{m_j\over m}\right)+2(m-r)\sum_{i=1}^nf_i{m_i\over m}- mr\sum_{i=1}^nf_i{m_i\over m}\\
&\le& 4m(r-1)\overline{f}+2(m-r)\overline{f}-mr\underline{f} = (4mr-2m-2r) \overline f - mr \underline f.
\end{eqnarray*}
We can now upper bound the quantity $\mathbb{E}\left[f(X)\right]$ from  (\ref{re0}) as follows:
\begin{eqnarray*}
\mathbb{E}\left[f(X)\right] \le f(x^*) +{(m-r)(4mr-2m-2r)\over r^2(m-1)(m-2)}(\overline{f}-\underline{f}).
\end{eqnarray*}
Together with Lemma \ref{lemmhdpoly}, this now  concludes the proof of Theorem \ref{thmcubsqf} (i).
\vspace{0.5cm}

We now consider the case (ii) when $f$ is a homogeneous square-free polynomial of degree $d$. Say, $f=\sum_{I\subseteq [n], |I|=d}f_I x^I$.
By Corollary \ref{momentcor},
 one has
\begin{eqnarray*}
\mathbb{E}[X^I] &=& {r^{\underline{d}}\over r^d}{\prod_{i\in I}m_i\over m^{\underline{d}}}={r^{\underline{d}}\over r^d}{m^d\over m^{\underline{d}}}\prod_{i\in I}{m_i\over m}
\end{eqnarray*}
and thus
\begin{eqnarray*}
\mathbb{E}[f(X)]=\sum_{I\subseteq [n],|I|=d}f_I\mathbb{E}[X^I] ={r^{\underline{d}}\over r^d}{m^d\over m^{\underline{d}}}f(x^*)={r^{\underline{d}}\over r^d}{m^d\over m^{\underline{d}}}f_{\Delta(n,m)}.
\end{eqnarray*}
Therefore,
$$\mathbb{E}[f(X)]-f_{\Delta(n,m)}=-\left( 1-{r^{\underline{d}}\over r^d}{m^d\over m^{\underline{d}}}\right)f_{\Delta(n,m)}\le -\left( 1-{r^{\underline{d}}\over r^d}{m^d\over m^{\underline{d}}}\right)\underline f \le \left( 1-{r^{\underline{d}}\over r^d}{m^d\over m^{\underline{d}}} \right)\left(\overline f - \underline f\right).$$
Here, for the last inequality we have used the fact that $\overline f\ge 0$
(since  $f(e_i)=0$ for any $i\in[n]$). Together with Lemma \ref{lemmhdpoly}, this concludes the proof of Theorem \ref{thmcubsqf} (ii).


\section{}\label{appimply}

Assume $f\in\mathcal{H}_{n,3}$ has a rational minimizer on $\Delta_n$ with  denominator $m\ge3$.

First we show how to derive Theorem \ref{thmklscubsqf} (i) for $r\ge 1+{m-1\over \sqrt{2m}-1}$ from our result in Theorem \ref{thmcubsqf} (i).

When $1+{m-1\over \sqrt{2m}-1}\le r\le m$, this follows directly from the fact that $\frac{(m-r)(4mr-2m-2r)}{r^2(m-1)(m-2)} \le {4\over r}-{4\over r^2}$.

\ignore{
that Theorem \ref{thmcubsqf} (i) implies Theorem \ref{thmklscubsqf} (i) when $r\ge 1+{m-1\over \sqrt{2m}-1}$.
Indeed, this follows directly from the fact that
$\frac{(m-r)(4mr-2m-2r)}{r^2(m-1)(m-2)} \le {4\over r}-{4\over r^2}$ when $ 1+{m-1\over \sqrt{2m}-1}\le r\le m.$}

Assume now $r>m\ge3$ and  $(k-1)m < r\le km$ for some integer $k\ge 2$. It suffices to show the inequality ${(km-r)(4kmr-2km-2r)\over r^2(km-1)(km-2)}\le {4\over r}-{4\over r^2}$ or, equivalently,
$$\varphi(r):= (2km-1)r^2+(4-6km)r-k^2m^2+6km-4\ge 0.$$

One can check that the function $\varphi(r)$ is monotonically increasing for $r\ge 1+{km-1\over 2km-1}$ and thus for $r\ge2$.
Hence it suffices to show that $\varphi((k-1)m+1)\ge 0$. If $m\ge3$ is fixed, then one can check that  $\varphi((k-1)m+1)$, as a function of $k$,  is monotonically increasing for $k\ge 2$.
Therefore, it suffices to show that $\varphi((k-1)m+1)\ge 0$ when $k=2$ and $m\ge3$.
One can now check that $\varphi((k-1)m+1)$ with $k=2$, as a function of $m$,  is monotonically increasing for $m\ge3$.
Finally, we can conclude that it suffices to show that  $\varphi((k-1)m+1)\ge 0$ when $k=2$ and $m=3$, which can be easily checked to hold.
Thus we have shown that $\varphi(r)\ge 0$ for any $r>m$.

\vspace{0.5cm}
To see that  Theorem \ref{thmcubsqf} (ii) implies Theorem \ref{thmklscubsqf} (ii), consider an integer
 $k\ge 1$ such that $(k-1)m<r \le km$ and observe
 that $1-{r^{\underline d}(km)^d\over r^d (km)^{\underline d}}\le 1-{r^{\underline d}\over r^d}$.

\vspace{0.5cm}
\section{}\label{appcor}

We prove Corollary \ref{corcub}.
Assume $m\ge 3$ is the  denominator for a rational global minimizer of $f$ over $\Delta_n$.
If $1\le r\le m$ then, by using Theorem \ref{thmcubsqf} (i), Lemma 2.1 and the inequality ${4mr-2m-2r\over r(m-2)}\le {4(m-1)\over m-2}$, we deduce that $$\rho_r(f) \le {(m-r)(4mr-2m-2r)\over r^2(m-1)(m-2)} \le {m^2\over r^2(m-2)}.$$

Assume now $r>m$ and  $(k-1)m < r\le km$ for some  integer $k\ge 2$. Then Theorem \ref{thmcubsqf} (i) implies
$$f_{\Delta(n,r)} - \underline f=f_{\Delta(n,r)} - f_{\Delta(n,km)}\le {(km-r)(4kmr-2km-2r)\over r^2(km-1)(km-2)}\left(\overline f - \underline f\right).$$
One can easily check that ${4kmr-2km-2r\over r(km-2)}\le 6$ which, together with Lemma \ref{lemkmr}, implies that  $\rho_r(f)\le{6m\over r^2}$.
This concludes the proof of Corollary \ref{corcub}.

\end{appendix}

\end{document}